\documentclass[11pt]{amsart}
\usepackage{amsmath}
\usepackage{amsfonts}
\usepackage{amsthm}
\usepackage{amssymb}
\usepackage{color}
\usepackage{todonotes}
\usepackage{enumitem}
\usepackage{tikz}
\usetikzlibrary{matrix,arrows,decorations.pathmorphing}
\usepackage{tikz-cd}

\usepackage[all,cmtip]{xy}
\DeclareMathOperator{\red}{red}
\newcommand{\of}{\circ}

\newtheorem{thm}{Theorem}[section]

\newtheorem{lem}[thm]{Lemma}
\newtheorem{prop}[thm]{Proposition}
\newtheorem{cor}[thm]{Corollary}
\theoremstyle{remark}

\newtheorem{rmk}[thm]{Remark}
\newtheorem{example}[thm]{Example}
\newtheorem{nota}[thm]{Notation}
\theoremstyle{definition}
\newtheorem{defn}[thm]{Definition}
\newtheorem{prop-def}[thm]{Proposition-Definition}
\newtheorem{q}{Question}

\DeclareMathOperator{\Pic}{Pic}
\DeclareMathOperator{\Div}{Div}

\DeclareMathOperator{\lct}{lct}
\DeclareMathOperator{\Supp}{Supp}
\DeclareMathOperator{\Proj}{Proj}
\newcommand{\Z}{\mathbb{Z}}
\newcommand{\Q}{\mathbb{Q}}

\newcommand{\PP}{\mathbb{P}}

\newcommand{\C}{\mathbb C}
\newcommand{\A}{\mathbb A}
\newcommand{\Oc}{\mathcal O}
\newcommand{\J}{\mathcal J}

\newcommand{\ceil}[1]{\left\lceil{#1}\right\rceil}
\newcommand{\floor}[1]{\left\lfloor{#1}\right\rfloor}

\allowdisplaybreaks

\presetkeys{todonotes}{fancyline}{}

\begin{document}

\title[Contribution of jumbing numbers by exceptional divisors]{Contribution of jumping numbers by exceptional divisors}

\author{Hans Baumers}

\author[Willem Veys]{Willem Veys\vspace{0.5cm}\\\tiny{with an appendix by Karen E.\ Smith and Kevin Tucker}}

\address{KU Leuven\\ Department of Mathematics \\ Celestijnenlaan 200B box 2400\\ BE-3001 Leuven \\ Belgium}

\email{Hans.Baumers@gmail.com, Wim.Veys@kuleuven.be}

\address{Department of Mathematics \\ University of Michigan \\ Ann Arbor \\ MI 48109 \\ USA}
\email{kesmith@umich.edu}
\address{Department of Mathematics, Statistics, and Computer Science\\ University of Illinois at Chicago\\Chicago\\  IL 60607 \\ USA}
\email{kftucker@uic.edu}

\thanks{The first author is supported by a PhD fellowship of the Research Foundation - Flanders (FWO)}

\begin{abstract}
We investigate some necessary and sufficient conditions for an exceptional divisor to contribute jumping numbers of an effective divisor on a variety of arbitrary dimension, inspired by the results for curves on surfaces by Smith and Thompson \cite{ST07} and Tucker \cite{Tuc10}. In particular, we construct an example of an exceptional divisor that is not contracted in the log canonical model, and does not contribute any jumping numbers.
\end{abstract}

\maketitle

\section{Introduction}\label{sec:intro}

%

The multiplier ideals $\J(X,\lambda D)$ associated to an effective divisor $D$ on an algebraic variety $X$ encode subtle information about the singularities of the pair $(X,D)$. They form a chain of $\Oc_X$-ideals $\J(X,\lambda D)$, which decrease when $\lambda$ increases, but remain the same after a slight increase of $\lambda$. The values of $\lambda$ where the multiplier ideals change are called the jumping numbers of the pair $(X,D)$. These geometric invariants where first studied in \cite{ELSV04}, but appeared earlier in different contexts, in \cite{Lib83}, \cite{LV90}, \cite{Vaq92} and \cite{Vaq94}. The smallest jumping number is the log canonical threshold. It has been studied thoroughly in e.g.\ \cite{Kol97} and \cite{Mus12}.


The multiplier ideals, and hence the jumping numbers, are computed using a log resolution of the pair $(X,D)$, so it is not a surprise that the exceptional divisors play an important role.
Smith and Thompson \cite{ST07}, and later Tucker \cite{Tuc10}, studied which exceptional divisors in an embedded resolution of $(X,D)$ are `relevant' for the computation of jumping numbers, introducing the notion `contribution of jumping numbers by exceptional divisors'. When $C$ is a curve on a surface $X$ with at most rational singularities, they found a geometrical characterization of exceptional divisors contributing jumping numbers by looking at the intersections with other components of the total transform of $C$ in the minimal resolution of $(X,C)$. They also prove that, if an exceptional divisor $E$ contributes a jumping number, it will always contribute the number $1-1/a$, where $a$ is the multiplicity of $E$ in the total transform of $C$. It turns out that, in this dimension, the contributing divisors coincide with the ones that are not contracted in the log canonical model of $(X,D)$ (see Definition \ref{def:lcmodel}).

The goal of this paper is to study to what extent these results can be generalized to higher dimensional varieties. In particular, we raise three questions, and formulate answers to each of them.

The first, and, in our opinion, the most important question, treats the relation between the exceptional divisors surviving in the log canonical model and those contributing jumping numbers, which was suggested by Smith and Thomspon in \cite{ST07}. We construct an example where an exceptional divisor does not contribute any jumping numbers, but survives in the log canonical model.

The second question is whether or not we can make conclusions about contribution of jumping numbers by a certain divisor, just by looking at the intersections with the other components of the total transform of $D$ in a log resolution. We encounter a big difference with the two-dimensional case here. In a log resolution of a curve on a surface with at most rational singularities, all exceptional divisors are projective lines, but in higher dimensions, there is a wide range of possibilities. We will study some specific cases where the intersection configuration contains enough information to decide whether or not an exceptional divisor contributes jumping numbers, and show that this does not hold in general by constructing a counterexample.

A final question we investigate is whether or not the number $1-1/a$ is always a jumping number if $E$ contributes. Here, $a$ is the multiplicity of $E$ in the total transform of $D$. Also here, the answer will be negative.

We start in Section \ref{sec:Basics} with introducing our basic concepts, such as multiplier ideals, jumping numbers and the notion of contribution of jumping numbers.
Next, in Section \ref{sec:birgeom}, we recall some definitions in birational geometry, and prove a contraction criterion for exceptional divisors in the log canonical model.
In Section \ref{sec:prelim}, we recall the results of Smith and Thompson \cite{ST07} and Tucker \cite{Tuc10} in the two-dimensional case. Also, we present some preliminary results in arbitrary dimension, which we use in section \ref{sec:positiveanswers} to show that the results of \cite{ST07} and \cite{Tuc10} still hold in higher dimensions for exceptional divisors that are not too complicated, for example exceptional divisors isomorphic to the projective space.
In Section \ref{sec:counterexample1}, we show by example that, in general, contribution of jumping numbers cannot be seen from the intersection configuration on the exceptional divisor.
Finally, in Section \ref{sec:couterexample2}, we give an example of an exceptional divisor that is not contracted in the log canonical model, and does not contribute any jumping numbers.

\subsection*{Acknowledgement} We would like to thank Christopher Hacon for his valuable help concerning the part on birational geometry, and in particular the contraction criterion (Lemma \ref{lem:contraction_in_lcmodel}).

\section{Basic notions}\label{sec:Basics}
We start with some definitions that will be used throughout this paper.
\begin{defn}
A \emph{variety} is an integral scheme of finite type over $\C$.
\end{defn}
\begin{defn}
A \emph{$\Q$-divisor} on a variety $X$ is an element of $\Div X\otimes_\Z\Q$. Equivalently, a $\Q$-divisor is of the form $F=\sum_{i=1}^n a_i F_i$, where the $F_i$ are irreducible Weil divisors on $X$ and the $a_i\in\Q$. A $\Q$-divisor $F$ is called $\Q$-Cartier if $mF$ is a Cartier divisor for some $m\in\Z$.
\end{defn}
\begin{defn}
If $F=\sum a_iF_i$ is a $\Q$-divisor, then the \emph{round down} of $F$ is $\floor{F}:=\sum\floor{a_i}F_i$.
\end{defn}
\begin{defn}
If $X$ is a normal variety and $D$ a $\Q$-divisor on $X$, then a \emph{log resolution} of $(X,D)$ is a proper, birational morphism $\pi:Y\to X$, such that
\begin{itemize}
\item $Y$ is smooth,
\item $\pi^{-1}(D\cup X_{Sing})$ is a strict normal crossings divisor,
\item $\pi$ defines an isomorphism outside $\pi^{-1}(D\cup X_{Sing})$.
\end{itemize}
\end{defn}


\begin{defn}
The \emph{relative canonical divisor} of a birational morphism of smooth varieties $\pi:Y\to X$ is \[K_\pi = K_Y-\pi^*K_X,\] where $K_X$ and $K_Y$ are the canonical divisors of $X$ and $Y$, respectively.
\end{defn}
\begin{rmk}
Although $K_Y$ and $K_X$ are only defined as divisor classes, we often consider $K_\pi$ as an effective divisor, since its divisor class contains a unique effective divisor supported on the exceptional locus of $\pi$.
\end{rmk}

Now we are ready to introduce multiplier ideals.
\begin{defn}
Let $X$ be a smooth variety and $D$ an effective divisor on $X$. Let $\pi:Y\to X$ be a log resolution of $(X,D)$. If $c$ is a positive rational number, we define the \emph{multiplier ideal} of $(X,D)$ with coefficient $c$ as \[\J(X,c D) := \pi_*\Oc_Y(K_\pi-\floor{c \pi^*D}).\]
\end{defn}
Note that, since $\pi_*\Oc_Y(K_\pi)=\Oc_X$, we have $\J(X,cD)\subseteq\Oc_X$ for all $c\in \Q_{>0}$, which justifies the name multiplier \emph{ideal}.
\begin{prop}[{\cite[Proposition 7.5]{EV92}}]
The multiplier ideal is independent of the chosen log resolution.
\end{prop}

From the definition of multiplier ideals, it is easy to see that a small increase of the coefficient $c$ does not affect the multiplier ideal. This gives rise to the concept of jumping numbers.

\begin{prop-def}
Let $D$ be an effective divisor on a smooth variety $X$.
There exists a chain of rational numbers \[0=\lambda_0<\lambda_1<\lambda_2<\dots<\lambda_i<\lambda_{i+1}<\dots\] satisfying
\begin{itemize}
\item for $i\in\Z_{\geq 0}$ and $c\in[\lambda_i,\lambda_{i+1})$, we have $\J(X,cD)=\J(X,\lambda_iD)$,
\item for $i\in\Z_{\geq0}$, $\J(X,\lambda_iD)\supsetneq\J(\lambda_{i+1}D)$.
\end{itemize}
The numbers $\lambda_i$, $i\geq 1$, are called the \emph{jumping numbers} of $(X,D)$.
\end{prop-def}

If $\pi:Y\to X$ is a log resolution of $(X,D)$, we can denote $\pi^*D=\sum_{i\in I} a_iE_i$ and $K_\pi = \sum_{i\in I} k_i E_i$, where $E_i$, $i\in I$, are the irreducible components of $\pi^{-1}(D)$. Then it is easy to see that the jumping numbers are contained in the set \[\left\{\left.\frac{k_i+n}{a_i}\right| i\in I,\, n\in \Z_{>0}\right\}.\]
The numbers in this set are called the \emph{candidate jumping numbers}. If $E_1,\dots,E_n$ are irreducible components of $\pi^{-1}(D)$, we say that a candidate jumping number $\lambda$ is a candidate for $E=\sum_{i=1}^n E_i$ if $\lambda a_i\in \Z$ for $i=1,\dots,n$. In contrast to multiplier ideals and jumping numbers, the notion of candidate jumping numbers depends on the chosen log resolution.

The smallest candidate jumping number however, does not depend on the chosen log resolution, and is always a jumping number. It is called the \emph{log canonical threshold} and we denote it by $\lct(X,D)$.

\vspace{2mm}
Now we list some basic properties.
First note that if $c\in\Q_{>0}$, we have
\begin{align*}
\J(X,(c+1)D) &= \pi_*\Oc_Y(K_\pi-\floor{c\pi^*D}-\pi^*D)\\
&= \J(X,cD)\otimes\Oc_X(-D)
\end{align*}
by the projection formula. Therefore, $c$ is a jumping number if and only if $c+1$ is a jumping number. This is actually a special case of Skoda's Theorem (see \cite[9.3.24]{LazII04}).

It is also easy to see that if $\lambda$ is a candidate jumping number for the strict transform of one of the components of $D$, it is always a jumping number. In particular, the positive integers are always jumping numbers for the pair $(X,D)$.


The following theorem is a useful tool for proving statements about multiplier ideals.
\begin{thm}[Local Vanishing, {\cite[Theorem 9.4.1]{LazII04}}]\label{thm:local_vanishing}
Let $D$ be a divisor on a smooth variety $X$, and $\pi:Y\to X$ a log resolution of $(X,D)$. Then for every $c\in\Q$ we have \[R^i\pi_*\Oc_Y(K_\pi-\floor{cD})=0\text{ for }i>0.\]
\end{thm}
Now we define contribution of jumping numbers by an exceptional divisor. This is a notion that indicates which exceptional divisors are responsible for the jumping numbers.
\begin{defn}[{\cite[Definition 2.1]{ST07}}]
Let $D$ be an effective divisor on a smooth variety $X$. Let $E$ be a reduced exceptional divisor (possibly reducible) in some log resolution $\pi:Y\to X$ of $(X,D)$, and $\lambda$ a candidate jumping number for $E$. We say $E$ \emph{contributes} $\lambda$ as a jumping number if
\[\J(X,\lambda D)\subsetneq \pi_*\Oc_Y(K_\pi-\floor{\lambda\pi^*D}+E).\]
\end{defn}
It is easy to see that this notion depends only on the valuations defined by the components of $E$. In particular it is independent of the choice of log resolution. 

\section{Birational geometry and the log canonical model}\label{sec:birgeom}

\begin{nota}
If $\pi:Y\to X$ is a birational morphism of normal algebraic varieties, and $D=\sum a_iD_i$ a $\Q$-divisor on $X$, where the $D_i$ are irreducible divisors, then we denote $\tilde D = \sum a_i\tilde D_i$, where $\tilde D_i$ is the strict transform of $D_i$ for every $i$.
\end{nota}

Let $X$ be a normal variety. Since the singular locus of $X$ has codimension at least 2, we can define the canonical divisor class $K_X$ by extending the canonical divisor on the non-singular locus of $X$ to all of $X$. Let $D$ be a $\Q$-divisor on $X$, such that $K_X+D$ is $\Q$-Cartier, and consider a log resolution $\pi:Y\to X$ of $(X,D)$ with exceptional prime divisors $E_i$, $i\in I$. 
If we choose appropriate representatives of $K_X$ and $K_Y$, then we can write
\[K_Y +\tilde D + \sum_{i\in I} E_i = \pi^*(K_X+D) + \sum_{i\in I} a(E_i,X,D)E_i\]
for some $a(E_i,X,D)\in \Q$. The number $a(E_i,X,D)$ is called the \emph{log discrepancy} of $E_i$ with respect to $(X,D)$. If $E$ is a prime divisor in a log resolution $f:Y\to X$, and $E'$ is a prime divisor in an other resolution $f':Y'\to X$ defining the same valuation, then $a(E,X,D)=a(E',X,D)$, so we can say that the discrepancy does not depend on the log resolution in which we consider a divisor.

If $D=\sum a_iD_i$, then it will be useful to extend the definition of log discrepancies to non-exceptional divisor by putting $a(D_i,X,D) = 1-a_i$ for all $i$ and $a(F,X,D)=0$ if $F$ is a prime divisor on $X$ different from the $D_i$.

\begin{defn}[{\cite[Definition 2.34]{KM98}}]\label{defn:logcanonical}
Let $X$ be a normal variety and $D$ a divisor on $X$ such that $K_X+D$ is $\Q$-Cartier.
We say that $(X,D)$ has \emph{log canonical singularities}, or simply that $(X,D)$ is \emph{log canonical} if $a(E,X,D)\geq 0$ for all exceptional divisors $E$ in all log resolutions of $(X,D)$. By \cite[Corollary 2.32]{KM98}, this is equivalent to $a(E_i,X,D)\geq 0$ for all exceptional divisors in a fixed resolution. Also, if $(X,D=\sum a_i D_i)$ is log canonical, then by \cite[Corollary 2.31]{KM98}, $a_i\leq 1$ for all $i$.
\end{defn}
\begin{defn}[{\cite[Definition 2.37]{KM98}}]
If $(X,D)$ is as in Definition \ref{defn:logcanonical}, and if $D=\sum a_iD_i$ with $0< a_i\leq 1$ for all $i$, then we say that $(X,D)$ is \emph{dlt} or \emph{divisorially log terminal} if there is a closed subset $Z\subset X$ such that $X\backslash Z$ is smooth, $D|_{X\backslash Z}$ is a simple normal crossings divisor, and there exists a log resolution $\pi:Y\to X$ of $(X,D)$ such that $\pi^{-1}(Z)$ has pure codimension one and $a(E,X,D)>0$ for every irreducible divisor $E\subseteq \pi^{-1}(Z)$.
\end{defn}

Log canonical pairs can also be described using multiplier ideals (see \cite[Definition 9.3.9]{LazII04}).
\begin{prop}
If $D$ is an effective divisor on a smooth variety $X$, then $(X,D)$ is log canonical if and only if \[\J(X,(1-\varepsilon)D) = \Oc_X\text{ for all }0<\varepsilon< 1.\] This happens if and only if $\lct(X,D)\geq 1$.
\end{prop}
\begin{proof}
If $\pi:Y\to X$ is a log resolution of $(X,D)$, and $\pi^*D=\sum_{i\in I}a_iE_i$, one can see that
\[\J(X,(1-\varepsilon)D) = \pi_*\Oc_Y\left(\sum_{i\in I}\ceil{a(E_i,X,D)-1+\varepsilon a_i}E_i\right),\]
and then the statement follows easily.
\end{proof}

%

\begin{defn}[{\cite[Definition 2.4]{Xu16}}, {\cite{HX13}}]\label{def:dltmodel}
If $X$ is a normal variety and $D=\sum a_iD_i$ a $\Q$-divisor on $X$, where the $D_i$ are distinct prime divisors and $0<a_i\leq1$, then a \emph{dlt model} of $(X,D)$ is a proper birational morphism $\phi_m:X_m\to X$ such that
\begin{enumerate}
\item the pair $(X_m, \tilde D+E_{\phi_m})$ is dlt, where $E_{\phi_m}$ is the reduced exceptional divisor of $\phi_m$, and
\item $K_{X_m}+\tilde D+E_{\phi_m}$ is $\phi_m$-nef, i.e., its restriction to any fibre of $\phi_m$ is nef.
\end{enumerate}
If $\pi:Y\to X$ is a log resolution of $(X,D)$ and $\phi_m:X_m\to X$ a dlt model, then there is an induced birational map $\phi:Y\dasharrow X_m$. We say that $X_m$ is a \emph{minimal dlt model of $(X,D)$ with respect to $\pi$} if $\phi^{-1}$ contracts no divisors, and $a(E,Y,\tilde D+E_{\pi})>a(E,X_m,\tilde D+E_{\phi_m})$ for all $\phi$-exceptional divisors $E\subset Y$. Here, $E_\pi$ denotes the reduced exceptional divisor of $\pi$.
\end{defn}

\begin{defn}[{\cite[Definition 2.1]{OX12}}]\label{def:lcmodel}
If $X$ and $D$ are as in Definition \ref{def:dltmodel}, then a \emph{log canonical model} of $(X,D)$ is a proper birational morphism $\phi_c:X_c\to X$ such that
\begin{enumerate}
\item the pair $(X_c,\tilde D+E_{\phi_c})$ is log canonical, where $E_{\phi_c}$ is the reduced exceptional divisor of $\phi_c$, and
\item $K_{X_c}+\tilde D+E_{\phi_c}$ is $\phi_c$-ample, i.e., its restriction to any fibre of $\phi_c$ is ample.
\end{enumerate}
\end{defn}
\begin{thm}[{\cite[Theorem 1.1, Proposition 2.1]{OX12}}, {\cite[Theorems 1.26, 1.32 and 1.34]{Kol13}}, {\cite[Lemma 2.4]{HX13}}]
If $X$ is a normal variety and $D=\sum a_iD_i$ a $\Q$-divisor on $X$, where the $D_i$ are distinct prime divisors and $0< a_i\leq1$, then there exists a unique log canonical model of $(X,D)$. A dlt model exists, but is not unique. However, if $\pi:Y\to X$ is a log resolution of $(X,D)$, then different minimal dlt models with respect to $\pi$ are isomorphic in codimension one.

If $X_m$ is a dlt model and $X_c$ the log canonical model, then there exists a morphism $X_m\to X_c$.
\end{thm}

\begin{defn}[see for example {\cite{LazI04}}]
%
Let $\pi:Y\to X$ be a morphism of varieties, with $X$ affine. If $L$ is a Cartier divisor on $Y$, $\Oc_Y(L)$ is the associated invertible sheaf, and $V\subseteq H^0(Y, \Oc_Y(L))$ is a linear subspace, then $|V|=\mathbb P(V)$ is a \emph{linear series} on $Y$ over $X$. If $V=H^0(Y,\Oc_Y(L))$, then we say $|V|$ is a \emph{complete linear series} over $X$ associated to $L$, also denoted $|L|$.

Let $|V|$ be a linear series with $V\subseteq H^0(Y,\Oc_Y(L))$, $E\subset Y$ a subvariety of $Y$, and $i:E\to Y$ the inclusion. The \emph{restricted linear series} over $X$ associated to $V$, denoted $|V|_E$, is $\mathbb P(i^*(V))$, where $i^*$ denotes the morphism $H^0(Y,\Oc_Y(L))\to H^0(E,i^*\Oc_Y(L))$.
\end{defn}
\begin{rmk}
The general definition of a linear series over $X$ is a subsheaf of $\pi_*\Oc_Y(L)$ (see for example \cite[Generalization 9.1.17]{LazII04}). However, on affine schemes, quasi-coherent sheaves are determined by their global sections. Therefore, this definition coincides with the classical definition of a linear series, not considered relative to $X$. We restrict to the affine case here, which is sufficient for us. The definition of the base locus below is also the same as the classical definition if $X$ is affine. However, the general definition does depend on $\pi$.
\end{rmk}
\begin{defn}
If $\pi:Y\to X$ is a morphism of normal varieties, with $X$ affine, and $|V|$ is a linear series on $Y$ over $X$, where $V\subseteq H^0(Y,\Oc_Y(L))$, then the base locus of $|V|$ over $X$ is 
\[B(|V|) = \bigcap_{s\in V} \Supp(\text{div}(s)),\]
where $\text{div}(s)$ denotes the divisor of zeroes of $s$.
This equals the closed set cut out by the image of
\[V\otimes_{\Oc_Y}\Oc_Y(-L)\to\Oc_Y.\]
If $L$ is a Cartier divisor on $Y$, then the stable base locus of $|L|$ over $X$ is defined as
\[{\bf B}(L) = \bigcap_{k\in\Z_{>0}} B(|kL|).\]
If $D$ is a $\Q$-divisor on $Y$, then we define ${\bf B}(D)={\bf B}(nD)$, where $n\in\mathbb Z_{>0}$ such that $nD$ is integral. This is independent of the choice of $n$ since ${\bf B}(L)={\bf B}(nL)$ for every Cartier divisor $L$ and every $n\in\Z_{>0}$ (see for example \cite[Proposition 2.1.21]{LazI04}).
\end{defn}
\begin{defn}
%
Let $\pi:Y\to X$ be a morphism of algebraic varieties, with $X$ affine. We say that a complete linear series $|L|$ on $Y$ over $X$ is \emph{big over $X$} if \[\limsup_{k\to\infty}\frac{h^0(F,\Oc_Y(kL)|_F)}{k^{\dim F}}>0,\]
where $F$ is a general non-empty fibre of $\pi$.

If $E\subset Y$ is a subvariety of $Y$, and $i:E\to Y$ is the embedding, then the restricted linear series $|L|_E$ is \emph{big} over $X$ if 
\[\limsup_{k\to\infty}\frac{\dim im(H^0(Y,\Oc_Y(kL))\to H^0(F,\Oc_Y(kL)|_F))}{k^{\dim F}}>0,\] where $F$ is a general non-empty fibre of the induced morphism $E\to X$.

If $D$ is a $\Q$-divisor on $Y$, then we say $|D|$, resp.\ $|D|_E$, is big over $X$ if so is $|nD|$, resp. $|nD|_E$, where $nD$ is an integral multiple of $D$.
\end{defn}

Given a log resolution $\pi:Y\to X$ of $(X,D)$, the following lemma tells us which exceptional divisors are contracted in the log canonical model. To the best of our knowledge, this result does not appear explicitly in the literature. We think it is of independent interest. A more general statement over a quasi-projective $X$ should hold, but that would lead us beyond the terminology and notation of the present paper.
Both the statement and the outline of the proof were kindly pointed out to us by Christopher Hacon.

\begin{lem}\label{lem:contraction_in_lcmodel}
Let $X$ be a normal affine variety and $D=\sum a_i D_i$ a $\Q$-divisor on $X$, with $0<a_i\leq 1$ for all $i$, such that $K_X+D$ is $\Q$-Cartier. Let $\pi:Y\to X$ be a log resolution of $(X,D)$.
Consider a minimal dlt model $X_m$ of $(X,D)$ with respect to $\pi$, and the log canonical model $X_c$ of $(X,D)$, in a diagram
\begin{center}\begin{tikzcd}
Y \arrow[dashed]{r}{\phi}[swap]{}\arrow{rd}{}[swap]{\pi}
&X_m \arrow{r}{\psi}[swap]{}\arrow{d}{\pi_m}[swap]{}
&X_c \arrow{ld}{\pi_c}[swap]{}\\
&X.
\end{tikzcd}\end{center}
Let $\Delta$ be the divisor $\Delta=\tilde D+\sum E_i$ on $Y$, where the $E_i$ are the irreducible exceptional divisors of $\pi$.

Then the divisors contracted by $\phi$ are precisely the divisors $E$ contained in ${\bf B}(K_Y+\Delta)$, 
and the divisors contracted by $\psi\circ\phi$ are the divisors $E$ such that the restricted linear series $|K_Y+\Delta|_E$ is not big over $X$.

\end{lem}
\begin{rmk}
The condition that $E$ is contained in ${\bf B}(K_Y+\Delta)$ is equivalent to $|k(K_Y+\Delta)|_E=\emptyset$ for all sufficiently divisible $k\in\Z_{>0}$. Indeed, both statements are equivalent to the vanishing of all sections of the $H^0(Y,k(K_Y+\Delta))$ on $E$.
\end{rmk}
\begin{proof}
First note that by the proof of Theorem 1.1 in \cite{OX12}, and in particular Lemma 2.8, we can assume that $(X_m,\phi_*\Delta)$ is a good minimal model over $(X,D)$, meaning that $K_{X_m}+\phi_*\Delta$ is $\pi_m$-semiample. Then the first part of the statement follows from \cite[Lemma 2.4]{HX13}. Indeed, $(Y,\Delta)$ is log canonical because it is a log resolution of $(X,D)$, hence it is dlt.


Note that if $E\subseteq {\bf B}(K_Y+\Delta)$, then $|K_Y+\Delta|_E $ cannot be big over $X$. Indeed, if $F$ is a fibre of $E\to X$ and $k\in\Z_{>0}$ is sufficiently divisible, then every section of $k(K_Y+\Delta)$ vanishes on $E$ and hence on $F$, so the image of \[H^0(Y,\Oc_Y(k(K_Y+\Delta)))\to H^0(F,\Oc_Y(k(K_Y+\Delta))|_F)\] is always zero.

Now let $E$ be a divisor that is not contracted by $\phi$, and denote $\phi_*E=E'\subset X_m$. Since by definition $\pi_*(k(K_Y+\Delta))=\pi_{m*}(k(K_{X_m}+\phi_*\Delta))$ for all sufficiently divisible $k\in\Z_{>0}$, we know that $|K_Y+\Delta|_E$ is big over $X$ if and only if so is $|K_{X_m}+\phi_*\Delta|_{E'}$.

Since $K_{X_m}+\phi_*\Delta$ is $\pi_m$-semiample, it follows that $K_{X_m}+\phi_*\Delta = \psi^*A$ for some $\pi_c$-ample $\Q$-divisor $A$ on $X_c$. Indeed, some multiple $k(K_{X_m}+\phi_*\Delta)$ is integral and base-point free, so it is the pullback of the very ample sheaf $\Oc(1)$ in $\Proj_X\bigoplus_{n\geq 0} \pi_{m*}\Oc_{X_m}(nk(K_{X_m}+\phi_*\Delta))$, which is precisely $X_c$ (see for example \cite[Theorem 3.52(1)]{KM98}). 

Let $F$ be a general non-empty fibre of the morphism $E'\to X$. We have the following diagram:

\begin{center}
\begin{tikzcd}
&F \arrow[two heads]{r}{p}\arrow[hook]{d}{i}
&\psi F\arrow[hook]{d}{j}\\
&X_m \arrow[two heads]{r}{\psi}
&X_c.
\end{tikzcd}
\end{center}

This induces, for any sufficiently divisible $k\in\Z_{>0}$, the following diagram:

%
%

\begin{center}
\begin{tikzcd}
&H^0(X_m,k\psi^*A) \arrow[leftarrow]{r}{\psi^*}\arrow{d}{i^*}
&H^0(X_c,kA)\arrow{d}{j^*}\\
&H^0(F,ki^*\psi^*A) \arrow[leftarrow]{r}{p^*}
&H^0(\psi F,kj^*A).
\end{tikzcd}
\end{center}

Note that $j^*$ is surjective for $k\gg 0$ by Serre vanishing, since the restriction of $A$ to $\psi F$ is ample. 
Also, $\psi^*$ is an isomorphism because $\psi_*\Oc_{X_m} = \Oc_{X_c}$, using the projection formula. Finally, since $p$ is surjective (and in particular dominant), $p^*$ is injective.

The statement of $|\psi^*A|_{E'}$ being big over $X$ is equivalent to \[\limsup_{k\to \infty}\frac{\dim i^*(H^0(X_m,k\psi^*A))}{k^{\dim F}}>0.\]  If $E'$ is contracted by $\psi$, i.e., $\dim (\psi F)<\dim F$, we have $im(i^*) = im(i^*\circ\psi^*) = im(p^*\circ j^*)\subseteq im(p^*)$. This implies that the dimension of $im(i^*)$ is at most $h^0(\psi F,kj^*A)$, which can grow with $k$ only as $k^{\dim (\psi F)}$. Therefore $|\psi^*A|_{E'}$ is not big over $X$.

Otherwise, if $E'$ is not contracted by $\psi$, then $\dim F = \dim (\psi F)$. Since $A$ is ample, we can take $n$ big enough such that $\dim H^0(\psi F,kj^*nA)\sim k^{\dim (\psi F)}=k^{\dim F}$. Then, because $p^*$ is injective, it follows that $|\psi^*A|_ {E'}$ is big over $X$.
\end{proof}
The following result should be well-known. We include the proof for completeness.
\begin{lem}\label{lem:contraction_in_lcmodel_curves}
Let $(X,D)$ be a normal variety and $D=\sum a_iD_i$ a $\Q$-divisor on $X$, with $0<a_i\leq 1$ for all i. Let $\pi:Y\to X$ be a log resolution of $(X,D)$ and $E$ a $\pi$-exceptional prime divisor on $Y$. Suppose that some Zariski open of $E$ is covered by curves $C$ whose classes belong to a fixed ray in the numerical cone of curves, and \[(K_Y+\Delta)\cdot C<0, \text{ (resp. } (K_Y+\Delta)\cdot C\leq 0\text),\]
where $\Delta=\tilde D+E_\pi$. Then $E$ is contracted in a dlt model with respect to $\pi$ (resp.\ the log canonical model) of $(X,D)$.
\end{lem}
\begin{proof}

Let $\phi:Y\dasharrow X_m$ be a dlt model with respect to $\pi$, and $\psi:X_m\to X_c$ the morphism onto the log canonical model. By for example \cite[1.9]{KK10} or \cite[Remark 2.7]{Bir12}, we know that $\psi^*(K_{X_c}+\psi_*\phi_*\Delta)=K_{X_m}+\phi_*\Delta$. Hence if $C$ is a curve on $X_m$, we have \[(K_{X_c}+\psi_*\phi_*\Delta)\cdot \psi_*C=\psi^*(K_{X_c}+\psi_*\phi_*\Delta)\cdot C=(K_{X_m}+\phi_*\Delta)\cdot C.\] Therefore, since $K_{X_c}+\psi_*\phi_*\Delta$ is ample over $X$, $C$ is contracted by $\psi$ if and only if $(K_{X_m}+\phi_*\Delta)\cdot C=0$.

So by running the minimal model program, we only have to check that after a flip or a divisorial contraction $f:Y'\dasharrow Y''$, either $E$ is contracted, or the transform of $E$ is still covered by $(K_{Y''}+f_*\Delta)$-negative curves. Indeed, by \cite[Lemma 3.38]{KM98}, the discrepancies of exceptional divisors over $X$ do not decrease after such a map, and hence the intersection with a movable curve cannot increase.
\end{proof}

\section{Preliminary results}\label{sec:prelim}
\subsection{The two-dimensional case}

%
%

When $C$ is a curve on a smooth surface $X$, and $E$ is an exceptional prime divisor in its minimal embedded resolution, then contribution of jumping numbers by $E$ was studied by Smith and Thompson in \cite{ST07}, and more generally by Tucker in \cite{Tuc10}. We have the following result.

\begin{thm}\label{thm:ST07}

Let $C$ be a curve on a smooth surface $X$, and $\pi:Y\to X$ the minimal embedded resolution of $(X,C)$. Let $E$ be an exceptional prime divisor of $\pi$, and set $d=E\cdot E^\circ$, where $E^\circ=(\pi^*C)_{red}-E$. Then the following are equivalent:
\begin{enumerate}
\item $E$ contributes jumping numbers to the pair $(X,C)$,\label{pn-1,1}
\item $E$ is not contracted in the log canonical model of $(X,C_{red})$,\label{pn-1,2}
\item $d\geq 3$.\label{pn-1,3}
\end{enumerate}
Moreover, in this case, $E$ contributes the jumping number $\lambda=1-\frac1a$, where $a$ is the multiplicity of $E$ in $\pi^*C$.
\end{thm}
The equivalence \ref{pn-1,1} $\Leftrightarrow$ \ref{pn-1,3} is the main result of \cite{ST07} (Theorem 3.1). The equivalence \ref{pn-1,2} $\Leftrightarrow$ \ref{pn-1,3} is well known (see for example \cite[Proposition 2.5]{Vey97}).
The implication \ref{pn-1,1} $\Rightarrow$ \ref{pn-1,2} holds in arbitrary dimension by Corollary \ref{cor:contrib=>no_contraction} below.

In the rest of the paper, we study to what extent the other equivalences can be generalized. We divide this problem in three questions.

Let $E$ be an exceptional prime divisor in a log resolution $\pi:Y\to X$ of an effective divisor $D$ on a smooth variety $X$. Write $\pi^*D=aE+\sum a_i E_i$, where the $E_i$ are the irreducible components of $\pi^{-1}(D)$ different from $E$.


\begin{q}
Does $E$ contribute jumping numbers if and only if it is not contracted in the log canonical model of $(X,D_{red})$?\label{q:cont=lcmodel}
\end{q}
\begin{q}
Can we draw conclusions about contribution by only looking at the intersection configuration on $E$ with other components of $\pi^*D$, i.e., is contribution determined by the class of $((\pi^*D)_{red}-E)|_E$ in $\Pic E$?
\label{q:cont<=intersections}
\end{q}
\begin{q}
If $E$ contributes jumping numbers, does it always contribute the number $1-1/a$?\label{q:same_jn}
\end{q}

The answer to Question \ref{q:same_jn} is negative, as can be seen from the following example.

\begin{example}\label{ex:not_1-1/a}
Let $D$ be the divisor given by $y(yz^2-x^2z+x^3+y^3)^2=0$ in $X=\mathbb A^3$. Blowing up at the origin first, with exceptional divisor $E_1$, followed by two line blow-ups, yields a resolution $\pi:Y\to X$, with $$K_\pi=2E_1+E_2+2E_3\text{, and }$$ $$\pi^*D = \tilde D + 7E_1+3E_2+6E_3,$$ where $\tilde D = 2D_1+D_2$ for prime divisors $D_1$ and $D_2$. One sees immediately that $\frac37$ is the log canonical threshold, so it is a jumping number contributed by $E_1$. However, $\frac67$ is not a jumping number by the following argument. The exceptional divisor $E_1$ is a projective plane, blown up at two infinitely near points. The Picard group is generated by the class of the pullback of a line in $\PP^2$, say $\ell$, the pullback of the first exceptional divisor, say $e_1$, and the second exceptional divisor, say $e_2$. Then we have $K_{E_1} = -3\ell+e_1+e_2$, $E_1|_{E_1} = -\ell$, $E_2|_{E_1} = e_1-e_2$, $E_3|_{E_1}=e_2$, $D_1|_{E_1}=3\ell-e_1-e_2$ and $D_2|_{E_1} = \ell-e_1-e_2$. So $K_{E_1}-\floor{\frac67\pi^*D}|_{E_1} =-e_2$, which is a class not containing an effective divisor. Hence, by Proposition \ref{prop:contribution_iff_globalsections} below, $\frac67$ is not a jumping number contributed by $E_1$. Since $E_1$ is the only divisor for which $\frac67$ is a candidate jumping number, we can even conclude that $\frac67$ is not a jumping number.

Using for example the algorithm of \cite{BD16}, we find that the complete list of jumping numbers in $(0,1]$ is $\frac37$, $\frac12$, $\frac56$ and $1$, which also yields the result.
\end{example}

\subsection{Preliminary results in arbitrary dimension}
Now we state some results that will be useful to prove statements about contribution of jumping numbers. An important tool is the following proposition, which appears for the two-dimensional case in \cite{ST07}, and in the general case, with a similar proof, in {\cite[Proposition 2.12]{BD16}}. We add the proof for completeness.
\begin{prop}\label{prop:contribution_iff_globalsections}
Let $D$ be an effective divisor on a smooth variety $X$, and let $E$ be an exceptional divisor in a log resolution $\pi:Y\to X$ of $(X,D)$. Denote by $i:E\to Y$ the embedding. Let $\lambda\in\Q_{>0}$ be a candidate jumping number for $E$. Then $E$ contributes $\lambda$ as a jumping number if and only if
\[\pi_*i_*i^*\Oc_Y(K_\pi-\floor{\lambda\pi^*D}+E)\neq 0.\]
If $\pi(E)$ is affine (for example when $E$ contracts to a point), this is equivalent to
\[H^0(E,i^*\Oc_Y(K_\pi-\floor{\lambda\pi^*D}+E))\neq 0.\]
If $E$ is prime, this means that $K_E-\floor{\lambda\pi^*D}|_E$ is equivalent to an effective divisor on $E$.
\end{prop}
\begin{proof}
Let $\lambda$ be a candidate jumping number for $E$ and consider the exact sequence
\begin{align*}
0\to \Oc_Y(K_\pi-\floor{\lambda\pi^*D})\to \Oc_Y&(K_\pi-\floor{\lambda\pi^*D}+E)\\&\to i_*i^*\Oc_Y(K_\pi-\floor{\lambda\pi^*D}+E)\to 0
\end{align*}
of sheaves on $Y$. Pushing forward through $\pi$, we obtain
\begin{align*}
0\to \J(X,\lambda D) \to \pi_*\Oc_Y(K_\pi&-\floor{\lambda \pi^*D}+E)\\ &\to \pi_*i_*i^*\Oc_Y(K_\pi-\floor{\lambda\pi^*D}+E)\to 0,
\end{align*}
where the last term is 0 by local vanishing (Theorem \ref{thm:local_vanishing}). So we see that $\lambda$ is a jumping number contributed by $E$ if and only if $\pi_*i_*i^*\Oc_Y(K_\pi-\floor{\lambda\pi^*D}+E)\neq0$. If $E$ is prime, we have $(K_\pi+E)|_E=K_E$ by adjunction, so the statement follows.
\end{proof}

As a consequence of this proposition, we have the following necessary condition for contributing jumping numbers.
\begin{cor}\label{cor:necessary_condition}
In the same setting as Proposition \ref{prop:contribution_iff_globalsections}, suppose $E$ is a prime divisor which is contracted to a point, and suppose that a divisor on $E$ is effective if and only if it is effective as a $\Q$-divisor. If $E$ contributes some jumping number $\lambda$ to $(X,D)$, then $K_E+E^\circ|_E$ is effective and non-zero in $\Pic E$, where $E^\circ=(\pi^*D)_{red}-E$.
\end{cor}
\begin{proof}
Denote $\pi^*D = \sum_{i\in I}a_iE_i + aE$, where $E$ and the $E_i$ are different prime divisors. If $\lambda$ is a candidate jumping number for $E$, then $\floor{\lambda\pi^*D} = \lambda\pi^*D-\sum_{i\in I} \{\lambda a_i\}E_i$. Hence, $\left(E^\circ +\floor{\lambda\pi^*D}\right)|_E = \sum_{i\in I}(1-\{\lambda a_i\})E_i|_E$ since $\pi^*D|_E=0$, and this is an effective $\Q$-divisor on $E$, different from the zero divisor, and hence an effective integral divisor.

If $E$ contributes $\lambda$ as a jumping number, then $K_E-\floor{\lambda\pi^*D}|_E$ is effective in $\Pic E$. Adding $E^\circ|_E+\floor{\lambda\pi^*D}|_E$ yields the result.
\end{proof}

The following theorem states that multiplier ideals can actually be computed using log canonical models instead of log resolutions. It is a special case of a theorem by Smith and Tucker, who have been so kind to provide the statement and the proof in the appendix to this paper (see Theorem \ref{LC-res}).
\begin{thm}\label{thm:ST}
Let $X$ be a smooth variety and $D$ an effective divisor on $X$. If $\phi_c:X_c\to X$ is the log canonical model of $(X,D_{red})$, and $\lambda\in\Q_{>0}$, then \[\J(X,\lambda D) = \phi_{c*}\Oc_{X_c}(K_{\phi_c}-\floor{\lambda\phi_c^*D}).\]
\end{thm}
\begin{cor}\label{cor:contrib=>no_contraction}
If an exceptional divisor contributes jumping numbers to the pair $(X,D)$, not all of its irreducible components can be contracted in the log canonical model.
\end{cor}


\section{Positive answers in specific situations}\label{sec:positiveanswers}

The proof of Theorem \ref{thm:ST07} builds on the fact that in the resolution of a curve on a smooth surface, every exceptional divisor is isomorphic to $\PP^1$, and hence has Picard group isomorphic to $\Z$. In the higher dimensional case, exceptional divisors can be more complicated. Therefore, a straightforward generalization of the proof of Theorem \ref{thm:ST07} is very unlikely. However, if we assume the exceptional divisor to be isomorphic to a specific, not too complicated variety, we can recover similar results, and find positive answers to our questions.

In the proofs in this section, we will use the results from \cite{Vey91}. These results are stated and proved for divisors on affine space, but this is used only to ensure that the pullback of a divisor restricted to an exceptional divisor $E$ is trivial in $\Pic E$. Therefore, these results also hold for prinicipal divisors on smooth varieties, or when $E$ is contracted to a point, which will be the setting in our propositions.

\begin{rmk}\label{rmk:minimal_resolution_caution}
We have to be careful in generalizing the statement of Theorem \ref{thm:ST07}, since a minimal resolution does not exist in higher dimensions. Therefore, we will assume in all of our statements that the log resolution is obtained by blowing up at centers that are either contained in the singular locus of $D$, or in the intersection of several components of the total transform of $D$. This does not give any limitations, because every pair has such a resolution (see \cite{Hir64}).
\end{rmk}

\subsection{Contribution by an exceptional divisor isomorphic to $\PP^{n-1}$}


The following proposition is the direct generalization of Theorem \ref{thm:ST07} to arbitrary dimensions. It can also be seen as a very special case of Proposition \ref{prop:pn-1_with_centers} below.

\begin{prop}\label{prop:pn-1}
Let $D$ be an effective divisor on a smooth $n$-dimensional variety $X$, with $n\geq 2$, and $\pi:Y\to X$ a log resolution of $(X,D)$. Let $E$ be an exceptional divisor of $\pi$ isomorphic to $\PP^{n-1}$, and let $d$ be the total degree in $E$ of the intersections of $E$ with the other components of $\pi^{-1}(D)$. Then the following are equivalent:
\begin{enumerate}
\item $E$ contributes jumping numbers to the pair $(X,D)$,\label{pn-1,1}
\item $E$ is not contracted in the log canonical model of $(X,D_{red})$,\label{pn-1,2}
\item $d\geq n+1$.\label{pn-1,3}
\end{enumerate}
Moreover, in this case, $E$ contributes the jumping number $\lambda=1-\frac1a$, where $a$ is the multiplicity of $E$ in $\pi^*D$.
\end{prop}
\begin{proof}
The implication \ref{pn-1,1} $\Rightarrow$ \ref{pn-1,2} is Corollary \ref{cor:contrib=>no_contraction}.

Write $\pi^*D = \sum_{i\in I} a_i E_i +aE$, where $\{E_i\mid i\in I\}$ are the components of $\pi^{-1}(D)$ different from $E$.
If $C$ is a line on $E$, then $(K_E+\sum_{i\in I} E_i|_E)\cdot C = d-n$. So by Lemma \ref{lem:contraction_in_lcmodel_curves}, if $d\leq n$, $E$ is contracted in the log canonical model. This proves \ref{pn-1,2} $\Rightarrow$ \ref{pn-1,3}.

It remains to prove that $E$ contributes $\lambda$ as a jumping number if $d\geq n+1$.
By Proposition \ref{prop:contribution_iff_globalsections}, it suffices to prove that $\deg(K_E-\floor{\lambda\pi^*D}|_E)\geq0$, or equivalently, that $\deg(\floor{\lambda\pi^*D}|_E)\leq -n$.

Let $E_j'$, $j\in J$, be the irreducible components of the intersections of $E$ with the other components of $\pi^{-1}(D)$. Denote $d_j=\deg(E_j')$ for every $j\in J$, so that $d=\sum_{j\in J}d_j$, and for every $j\in J$ denote $a_j=a_i$, where $i\in I$ is the index such that $E_j'$ is a component of $E_i\cap E$.
By \cite{Vey91}, we have
\begin{align*}
\deg(aE|_E) &= -\sum_{j\in J}a_jd_j,\\
\sum_{j\in J} d_ja_j &= \left(1+\sum_{j\in J} d_j m_j\right)a,
\end{align*}
where $m_j$ is the number of times that the strict transform of $E_j'$ on $E$ has been used as center of a blow-up in the resolution process.
This implies that
\begin{align*}
\deg\left(\floor{\lambda\pi^*D}|_E\right) &= \deg\left(\sum_{i\in I}\floor{a_i-\frac{a_i}{a}}E_i|_E+(a-1)E|_E\right)\\
&= \sum_{j\in J}\floor{a_j-\frac{a_j}{a}}d_j + (a-1)\deg(E|_E)\\
&= -\sum_{j\in J}\ceil{\frac{a_j}a}d_j - \deg(E|_E)\\
&= 1-\sum_{j\in J}\left(\ceil{\frac{a_j}a}-m_j\right)d_j.
\end{align*}
Now note that our assumptions on the resolution (Remark \ref{rmk:minimal_resolution_caution}) imply that $a_j>m_ja$ for every $j\in J$,  and hence $\ceil{\frac{a_j}a}-m_j\geq1$. Therefore, if $d\geq n+1$, we have $\deg(\floor{\lambda\pi^*D}|_E)\leq -n$. This completes the proof.
\end{proof}

\subsection{$\PP^{n-1}$ blown up at some centers on a hyperplane}

%
%

Throughout this section, we prove the following proposition.
\begin{prop}\label{prop:pn-1_with_centers}

Let $D$ be an effective divisor on a smooth $n$-dimensional variety $X$, with $n\geq2$, and $\pi:Y\to X$ a log resolution of $(X,D)$. Let $E$ be an exceptional divisor of $\pi$ isomorphic to $\PP^{n-1}$, blown up at some centers $Z_l$, $l\in L$, all contained in the same hyperplane $H$. Assume that $E$ is created by a point blow-up, and denote $\dim Z_l=k_l$.

Denote by $d$ the total degree of the intersections of $E$ with other components of the total transform of $D$ at the moment of the creation of $E$, and by $\mu_l$ the total multiplicity of these components at $Z_l$ for every $l\in L$. Then the following are equivalent:
\begin{enumerate}
\item $E$ contributes jumping numbers to the pair $(X,D)$,\label{pn-1,1}
\item $E$ is not contracted in the log canonical model of $(X,D_{red})$,\label{pn-1,2}
\item $d\geq n+1$ and $d-\mu_l\geq k_l+2$ for every $l\in L$.\label{pn-1,3}
\end{enumerate}
Moreover, in this case, $E$ contributes the jumping number $\lambda=1-\frac1a$, where $a$ is the multiplicity of $E$ in $\pi^*D$.
\end{prop}
\begin{rmk}An example of this situation is a projective plane blown up at two points, or, more generally, at any number of points on a fixed line.\end{rmk}
Note that \ref{pn-1,1} $\Rightarrow$ \ref{pn-1,2} is Corollary \ref{cor:contrib=>no_contraction}.
Before proving this proposition, we introduce some notations. 
Denote $\pi^*D = \sum_{i\in I} a_iE_i + aE$, where $\{E_i\mid i\in I\}$ are the components of $\pi^{-1}(D)$ different from $E$.

The Picard group of $E$ is isomorphic to $\Z\oplus \bigoplus_{l\in L}\Z$, with generators $h$, the pullback of a hyperplane in $\PP^{n-1}$, and $e_l$, $l\in L$, the exceptional divisors of the blow-ups at the $Z_l$.

Let $E_j'$, $j\in J$, be the irreducible components of the intersections of $E$ with the other components of $\pi^{-1}(D)$.

Since every exceptional divisor on $E$, created by blowing up at the $Z_l$, is the intersection of $E$ with one of the other components of $\pi^{-1}(D)$ (see \cite{Vey91}), we can view $L$ as a subset of $J$. Then, for $j\in J':=J\backslash L$ and $l\in L$, we can define $d_j$ and $\mu_{jl}$ so that we have the following equalities in $\Pic E$:
\begin{align*}
E_j' &= d_j h - \sum_{l\in L} \mu_{jl}e_l, \text{ and}\\
E_l' &= e_l.
\end{align*}
Note that with these notations $d=\sum_{j\in J'}d_j$ and $\mu_l = \sum_{j\in J'}\mu_{jl}$ for all $l\in L$. Also, as in the proof of Proposition \ref{prop:pn-1}, denote $a_j=a_i$ for every $j\in J$, where $i\in I$ is the index such that $E_j'$ is a component of $E_i\cap E$.

The canonical divisor of $E$ is given by \begin{equation}\label{comp:K_E}
K_E = -nh+\sum_{l\in L}(n-k_l-2)e_l.
\end{equation}

For every $l\in L$, the blow-up at $Z_l$ on $E$ arises from a blow-up in the ambient space. Denote the center of this blow-up by $C_l$, such that $E\cap C_l=Z_l$ (at this stage of the resolution).

\subsubsection{Contraction in the log canonical model}
Consider the family of strict transforms or pullbacks of lines $C$ in $\PP^{n-1}$, not intersecting any of the $Z_l$. Then we see by Lemma \ref{lem:contraction_in_lcmodel_curves} that $E$ is contracted in the log canonical model if $(K_E+\sum_{j\in J} E_j')\cdot C \leq 0$, which is equivalent to $d\leq n$. (This follows from $h\cdot C=1$ and $e_l\cdot C=0$ for all $l\in L$.)

Now fix one of the centers $Z_l$ and consider the family of strict transforms of lines $C$ in $\PP^{n-1}$ intersecting $Z_l$ transversally, and none of the other $Z_l$. (If $Z_l$ is a point, intersecting $Z_l$ transversally just means that the line contains $Z_l$.) Then we see that $E$ is contracted in the log canonical model if $(K_E+\sum_{j\in J} E_j')\cdot C\leq 0$, which is equivalent to $d-\mu_l\leq k_l+1$. (This follows from $h\cdot C = 1$, $e_l\cdot C = 1$, and $e_{l'}\cdot C = 0$ for $l'\neq l$.) This proves \ref{pn-1,2} $\Rightarrow$ \ref{pn-1,3}.

\subsubsection{Contribution of jumping numbers}
Now we show the implication \ref{pn-1,3} $\Rightarrow$ \ref{pn-1,1}, using Proposition \ref{prop:contribution_iff_globalsections}. 
We have
\begin{align}
-\floor{\lambda\pi^*D}|_E &= -(\pi^*D)|_E + \sum_{i\in I}\ceil{\frac{a_i}{a}}E_i|_E + E|_E\nonumber\\
&= \sum_{i\in I}\ceil{\frac{a_i}{a}}E_i|_E+E|_E\nonumber\\
&= \sum_{j\in J}\ceil{\frac{a_j}{a}}E_j'+E|_E\nonumber\\
&= E|_E + \sum_{j\in J'} \ceil{\frac{a_j}{a}}d_j h + \sum_{l\in L} \left(\ceil{\frac{a_l}{a}}-\sum_{j\in J'} \ceil{\frac{a_j}{a}}\mu_{jl}\right)e_l.\label{comp:floorD}
\end{align}

Moreover, as in the proof of Proposition \ref{prop:pn-1}, from \cite{Vey91} we have
\begin{align}
\sum_{j\in J'}d_ja_j &= \left(1+\sum_{j\in J'}d_jm_j\right)a,\label{comp:diai}\\
a_l &= \sum_{j\in J'}\mu_{jl}a_j + \left(m_l-\sum_{j\in J'}\mu_{jl}m_j^{(l)}+\delta_l\right)a\text{ for all $l\in L$, and }\label{comp:a_j}\\
aE|_E &= -\sum_{j\in J}a_jE_j' = -\left(\sum_{j\in J'}d_ja_j\right)h - \sum_{l\in L}\left(a_l-\sum_{j\in J'}\mu_{jl}a_j\right)e_l,\nonumber
\end{align}
where $m_j$ (respectively $m_j^{(l)}$) is the number of times that the strict transform of $E_j'$ on $E$ has been used as the center of a blow up after the creation of $E$ (respectively after blowing up at $C_l$), and $\delta_l=1$ if $Z_l=C_l$ (or equivalently, $C_l\subset E$), and $\delta_l=0$ otherwise.

Hence, using (\ref{comp:diai}) and (\ref{comp:a_j}) we obtain
\begin{align}
E|_E &= -\left(1+\sum_{j\in J'} d_jm_j\right)h - \sum_{l\in L}\left(m_{l}-\sum_{j\in J'} \mu_{jl}m_j^{(l)}+\delta_l\right)e_l\label{comp:E|_E}
\end{align}
because $\Pic E$ is torsion free. 

Combining (\ref{comp:K_E}), (\ref{comp:floorD}) and (\ref{comp:E|_E}), we have
\begin{align*}
K_E - \floor{\lambda\pi^*D}|_E = &\left(-n + \sum_{j\in J'} \ceil{\frac{a_j}{a}}d_j -\left(1+\sum_{j\in J'} d_jm_j\right)  \right)h\\
&+\sum_{l\in L}\left(n-k_l-2 + \ceil{\frac{a_{l}}{a}}-\sum_{j\in J'} \ceil{\frac{a_j}{a}}\mu_{jl} \right.\\&\left.\hspace{1cm}- \left(m_{l}-\sum_{j\in J'} \mu_{jl}m_j^{(l)}+\delta_l\right) \right)e_l\\
= &\left(-n -1+ \sum_{j\in J'} \left(\ceil{\frac{a_j}{a}}-m_j\right)d_j \right)h\\
&+\sum_{l\in L}\left(n-k_l-2 + \ceil{\sum_{j\in J'} \frac{\mu_{jl}a_j}{a}}-\sum_{j\in J'} \ceil{\frac{a_j}{a}}\mu_{jl}\right)e_l,
\end{align*}
using (\ref{comp:a_j}) to rewrite $\ceil{\frac{a_l}{a}}$.

If $\tilde H$ is the strict transform of $H$, we have $\tilde H = h-\sum_{l\in L}e_l$ in $\Pic E$. Then we obtain
\begin{align*}
K_E - \floor{\lambda\pi^*D}|_E = &\left(-n -1+ \sum_{j\in J'} \left(\ceil{\frac{a_j}{a}}-m_j\right)d_j \right)\tilde H\\
&+\sum_{l\in L}\left( -k_l-3 + \sum_{l\in L'} \left(\ceil{\frac{a_j}{a}}-m_j\right)d_j \right.\\&\left.\hspace{1cm}+ \ceil{\sum_{j\in J'} \frac{\mu_{jl}a_j}{a}}-\sum_{j\in J'} \ceil{\frac{a_j}{a}}\mu_{jl} \right)e_l.
\end{align*}

So we see that $E$ contributes $\lambda=1-\frac1a$ as a jumping number if
\begin{align*}
\sum_{j\in J'} \left(\ceil{\frac{a_j}{a}}-m_j\right)d_j&\geq n+1,\text{ and}\\
\sum_{j\in J'} \left(\ceil{\frac{a_j}{a}}-m_j\right)d_j &\geq k_l+3+\sum_{j\in J'} \ceil{\frac{a_j}{a}}\mu_{jl}-\ceil{\sum_{j\in J'} \frac{\mu_{jl}a_j}{a}} \text{ (for all }l).
\end{align*}
Since $\frac{a_j}{a} > m_j$ for all $j\in J'$, we have $\ceil{\frac{a_j}{a}}\geq m_j+1$ and $\ceil{\sum_{j\in J'}\frac{\mu_{jl}a_j}{a}} \geq \sum_{j\in J'} \mu_{jl}m_j +1$ for every $l$, which implies that $\lambda$ is a jumping number contributed by $E$ if
\begin{align*}
\sum_{j\in J'} d_j&\geq n+1,\text{ and}\\
\sum_{j\in J'} d_j &\geq k_l+2+\sum_{j\in J'} \mu_{jl} \text{ (for all }l),
\end{align*}
i.e., if
\begin{align*}
d&\geq n+1,\text{ and}\\
d-\mu_l &\geq k_l+2 \text{ (for all }l).
\end{align*}
Indeed, if $\sum_{j\in J'} d_j\geq n+1$, then \[\sum_{j\in J'}\left(\ceil{\frac{a_j}{a}}-m_j\right)d_j\geq \sum_{j\in J'} d_j\geq n+1,\] and if $\sum_{j\in J'} (d_j-\mu_{jl}) \geq k_l+2$ for some $l\in L$, then \[\sum_{j\in J'}\left(\ceil{\frac{a_j}{a}}-m_j\right)(d_j-\mu_{jl})\geq\sum_{j\in J'} (d_j-\mu_{jl}) \geq k_l+2,\] and consequently
\begin{align*}
\sum_{j\in J'} \left(\ceil{\frac{a_j}{a}}-m_j\right) d_j &\geq k_l+2+\sum_{j\in J'} \ceil{\frac{a_j}{a}}\mu_{jl} - \sum_{j\in J'}\mu_{jl}m_j\\
&\geq k_l+3+\sum_{j\in J'} \ceil{\frac{a_j}{a}}\mu_{jl} - \ceil{\sum_{j\in J'}\frac{\mu_{jl}a_j}{a}}.
\end{align*}
This finishes the proof of Proposition \ref{prop:pn-1_with_centers}.

\section{A counterexample to Question \ref{q:cont<=intersections}}\label{sec:counterexample1}
If $\dim X=3$, then besides $\PP^2$, and $\PP^2$ blown up at some distinct points on a line, the easiest case is when $E$ is an exceptional divisor isomorphic to $\PP^2$, blown up at two infinitely near points.

So suppose we have such an $E$. Denote by $d$ the degree of the intersections of $E$ with other components of the total transform of $D$ after the moment of its creation, and by $\mu_1$ and $\mu_2$ the multiplicity of these intersections at the first, respectively the second point.

As in the proof of Proposition \ref{prop:pn-1_with_centers}, one can show that $E$ contributes the jumping number $1-1/a$ if $d\geq 4$, $d-\mu_1\geq 2$ and $2d-\mu_1-\mu_2\geq 5$, where $a$ is the multiplicity of $E$ in the total transform of $D$. Also, using Lemma \ref{lem:contraction_in_lcmodel_curves}, $E$ is contracted in the log canonical model if $d\leq 3$ (if we look at a general line), $d-\mu_1\leq 1$ (if we consider a line through the first point) or $2d-\mu_1-\mu_2\leq3$ (if we look at a degree 2 curve through both points). Hence, in these cases, $E$ does not contribute any jumping numbers.

In contrast with the previous results, this does not cover all the possibilities. Concretely, the cases where $d\geq 4$, $d-\mu_1\geq 2$ and $2d-\mu_1-\mu_2=4$ are still open. Example \ref{ex:not_1-1/a} already shows that the case $d=4$, $\mu_1=\mu_2=2$ cannot be classified in one of the two options listed above. The following examples show even more: equal intersection configurations can lead to different statements about contribution (and contraction in the log canonical model). Hence, the answer to Question \ref{q:cont<=intersections} is negative in general.


\begin{example}\label{ex:contr}
Let $D$ be the divisor given by $(xy^2-z^2)(x+z)=0$ in $X=\A^3$. We can construct a log resolution by blowing up at the origin first, with exceptional divisor $E_1$, followed by blowing up at the intersection of $E_1$ with the two components of $D$, with exceptional divisor $E_2$, further at the singular line on the strict transform of the first component of $D$, with exceptional divisor $E_3$, and then resolving the tangency of $E_1$ with the strict transform of the first component of $D$, using two more blow-ups, with exceptional divisors $E_4$ and $E_5$. If $\pi:Y\to X$ is the composite of these blow-ups, we  have 
\begin{align*}
\pi^*D &= \tilde D + 3E_1 + 6E_2 + 2E_3 + 4E_4 + 8E_5,\\
K_\pi &= 2E_1 + 4E_2 + E_3 + 3E_4 + 6E_5.
\end{align*}
One can see that $E_2$ is a projective plane, blown up at two infinitely near points. 
At the moment of its creation, the other components of the total transform of $D$ intersect $E_2$ in a curve of degree 2, a line tangent to this curve, and a line intersecting these curves transversally. This means we are in the situation $d=4$, $\mu_1=\mu_2=2$.
One can see immediately that $\frac56$ is the log canonical threshold, contributed as a jumping number by $E_2$.
\end{example}

\begin{example}\label{ex:D5}
Now consider the $D_5$-singularity, given by $yz^2+x^2-y^4=0$. We construct a resolution $\pi:Y\to X$ by blowing up at the origin, with exceptional divisor $E_1$, followed by blowing up at the origin of the second chart, with exceptional divisor $E_2$, blowing up at the intersection of $E_1$, $E_2$ and the strict transform of $D_5$, with exceptional divisor $E_3$, and then twice at the intersection of $E_1$ with the strict transform of $D_5$, with exceptional divisors $E_4$ and $E_5$. Then $E_3$ is a projective plane, blown up at two infinitely near points, and the intersection configuration is the same as in the previous example. However, since $D_5$ is a log canonical singularity, it has no jumping numbers in $(0,1)$. (This can also be verified using the algorithm of \cite{BD16}.) We can conclude that contribution of jumping numbers by an exceptional divisor cannot be decided by only looking at the intersection configuration.
\end{example}
\begin{rmk}
We can say even more. In Examples \ref{ex:contr} and \ref{ex:D5}, the exceptional divisors we considered are even created in a similar way, i.e., blowing up at a point first, and then twice at a line intersecting the divisor transversally.
\end{rmk}

\section{A counterexample to Question \ref{q:cont=lcmodel}}\label{sec:couterexample2}

\begin{example}
Consider the divisor $D=\left\{(zy+x^2)^2+x^3y+xy^3=0\right\}$ in $X=\mathbb A^3$. We blow up at the origin first, and call the exceptional divisor $E_0$. Then, after four additional blow-ups centered in a line, corresponding to the minimal resolution of the singular curve $\tilde D\cap E_0$ in $E_0$, we obtain a log resolution $\pi:Y\to X$. We have
\begin{align*}
K_\pi &= 2E_0 + E_1+2E_2+3E_3+6E_4,\\
\pi^*D &= \tilde D + 4E_0+2E_1+4E_2+5E_3+10E_4,
\end{align*}
where $\tilde D$ denotes the strict transform of $D$.

The only candidate jumping numbers for $E_0$ in $(0,1]$ are $\frac34$ and $1$. Using Proposition \ref{prop:contribution_iff_globalsections}, one can show that they are not contributed by $E_0$ (similarly as in Example \ref{ex:not_1-1/a}).
With for example the algorithm of \cite{BD16}, one can compute that the jumping numbers are in fact the numbers in the set $\left\{\frac{7}{10}, \frac{9}{10},1\right\} + \mathbb Z_{\geq0}$, and then the statement for $\frac34$ also follows.

We show that $E_0$ is not contracted in the log canonical model using Lemma \ref{lem:contraction_in_lcmodel}. Since $D\sim 0$ on $X$, we have $\pi^*D\sim 0$ on $Y$, and therefore \[\tilde D \sim -4E_0-2E_1-4E_2-5E_3-10E_4.\] Hence, \[K_\pi+\tilde D + \sum_{i=0}^4 E_i \sim -E_0-E_2-E_3-3E_4.\] If $G$ is the strict transform of a general plane in $X$ through the origin, we have $G\sim -E_0$. Similarly, if $F$ is the strict transform of the divisor $\{y=0\}$, one can compute that \[F\sim -E_0-E_1-2E_2-2E_3-4E_4.\] Therefore, \[K_\pi+\tilde D + \sum_{i=0}^4 E_i \sim_\Q \frac14G+\frac34F+\frac34E_1+\frac12E_2+\frac12E_3.\] 

Now, for $k\in\Z_{>0}$, \[h^0(E_0,k(G+3F+3E_1+2E_2+2E_3)|_{E_0})\geq h^0(E_0,kG|_{E_0})={{k+2}\choose k},\]  hence $\left|K_Y+\tilde D+\sum_{i=0}^4E_i\right|_{E_0}$ is big over $X$, and $E_0$ is not contracted in the log canonical model.
\end{example}

\appendix
\section[Multiplier ideals from an LC-resolution]{Multiplier ideals from an LC-resolution \\ (by Karen E.\ Smith\protect\footnote{The first author was partially supported by NSF Grant DMS \#1501625.} and Kevin Tucker\protect\footnote{The second author was partially supported by NSF Grant DMS \#1602070 and a fellowship from the Sloan foundation.})}\label{appendix}

We denote by $S$ a smooth complex variety, and $C$ an effective divisor on $S$.

\begin{defn}
Suppose $X$ is a normal complex variety,  $f: X \to S$ a proper birational morphism, and let $\Delta = (f^{*}C)_{\red}$.  Then $f \colon X \to S$ is an \emph{LC-resolution} of $(S,C)$ if $K_{X} + \Delta$ is $\Q$-Cartier, and for some (equivalently all) dominating log resolutions of $(S,C)$
\[  \xymatrix{ 
 X'   \ar[r]_{\theta} \ar@/^1pc/[rr]^{f'} &   X   \ar[r]_{f} &   S      
} \]
we have
\[
K_{X'} + \Delta' \geq \theta^{*}(K_{X} + \Delta)
\]
where $\Delta' = (f'^{*}C)_{\red}$.
\end{defn}

In other words, a proper birational morphism $f \colon X \to S$ is an LC-resolution if and only if $X$ is normal, 
 $(X, \Delta = (f^*C)_{\red})$ is log canonical pair,  and $X \setminus \Delta$ has canonical singularities.  In practice, one generally restricts to LC-resolutions which are an isomorphism outside $C$ (or even where $C$ is singular), so that the last requirement is automatic.  For more information on the types of singularities involved and a number of related constructions, see  \cite{Kol13} (particularly Section 1.4).

\begin{thm}\label{LC-res}
If $f \colon X \to S$ is an LC-resolution of $(S,C)$ and $\lambda \in \Q_{>0}$, then
\[
\J(S, \lambda C) = f_{*} \Oc_{X}( K_{f} - \lfloor \lambda f^{*}C \rfloor ) \, .
\]
In other words, the multiplier ideals of $(S,C)$ can be computed from any LC-resolution.
\end{thm}

\begin{proof}
Choose a dominating resolution $f' \colon X' \to S$ as above.  Let $K_{f'}$ denote the (unique exceptionally supported) divisor $K_{X'}-f^{'*}K_S$. Since we have
\[
\theta_{*}(K_{f'} - \lfloor \lambda f'^{*}C \rfloor )= K_{f} - \lfloor \lambda f^{*}C \rfloor \, ,
\]
it follows immediately that
\[
\J(S, \lambda C) = f_{*}\theta_{*} \Oc_{X'}(K_{f'} - \lfloor \lambda f'^{*}C \rfloor) \subseteq f_{*} \Oc_{X}( K_{f} - \lfloor \lambda f^{*}C \rfloor ) \, .
\]
For the opposite inclusion, we may assume that $S$ is affine.  Suppose 
$
\varphi \in  H^{0}(X, K_{f} - \lfloor \lambda f^{*}C \rfloor)
$, so that $\varphi \in K(X)$ and
\[
\text{div}(\varphi) + K_{f} - \lfloor \lambda f^{*}C \rfloor \geq 0. \,
\]
Write $f^*(\lambda C) =  \lfloor \lambda f^{*}(\lambda C) \rfloor + \{ \lambda f^{*}(\lambda C) \}$, where 
$\{ D \}$ denotes the fractional part of a divisor $D$, so that
\begin{equation}\label{eq1}
\text{div}(\varphi) + K_{f} - f^*(\lambda C)  + \{ \lambda f^{*}(\lambda C)\}  \geq 0. 
\end{equation}

Choose a rational number $\epsilon>0$   sufficiently small  so that $\Delta' - \epsilon f'^{*}(\lambda C)  \geq \{ \lambda f'^{*}C \}$.  Pushing forward by $\theta$, this also implies $\Delta - \epsilon f^{*}(\lambda C)  \geq \{ \lambda f^{*}C \}$.
Therefore, in light of (\ref{eq1}),
\[
\text{div}(\varphi) - f^{*}K_{S}  + K_{X} + \Delta - (1+\epsilon)f^{*}(\lambda C) \geq 0
\]
and hence also (recalling that $K_{X} + \Delta$ is $\Q$-Cartier)
\[
\text{div}( \varphi \of \theta) - f'^{*}K_{S} + \theta^{*}(K_{X} + \Delta) - (1 + \epsilon)f'^{*}(\lambda C) \geq 0 \,.
\]
Since $X$ is an LC-resolution, we have $K_{X'} + \Delta' \geq \theta^{*}(K_{X} + \Delta)$.  Thus,
\[
\text{div}( \varphi \of \theta)  - f'^{*}K_{S} + K_{X'} +\Delta' - (1 + \epsilon)f'^{*}(\lambda C) \geq 0
\]
\[
\text{div}( \varphi \of \theta) + K_{f'} - \lfloor f'^{*}(  \lambda C) \rfloor + (\Delta' - \epsilon f'^{*}( \lambda C) - \{ \lambda f'^{*}C \}) \geq 0
\]
Taking the integer part of the left side yields
\[
\text{div}( \varphi \of \theta) + K_{f'} - \lfloor f'^{*}( \lambda C) \rfloor \geq 0 \, ,
\]
so that $\varphi \in f'_*\mathcal O_{X'}(K_{f'} -  \lfloor f'^{*}( \lambda C) \rfloor ) = \J(S, \lambda C). $ The proof is complete.
\end{proof}

\begin{rmk}[Log Canonical Models]
As a corollary to the established results of the log minimal model program, Odaka and Xu \cite{OX12} have verified the existence of a unique LC-resolution $f_{\mathrm{lc}} \colon X_\mathrm{lc} \to S$ so that if $\Delta_{\mathrm{lc}} = (f_{\mathrm{lc}}^*C)_{\mathrm{red}}$ then $K_{X_{\mathrm{lc}}} + \Delta_{\mathrm{lc}}$ is $f_{\mathrm{lc}}$-ample.
\end{rmk}

\addcontentsline{toc}{chapter}{Bibliography}
\bibliographystyle{amsalpha}
\bibliography{Bibliography}

\end{document}